\documentclass[12pt,reqno]{amsart}
\usepackage[top=1in, bottom=1in, left=1in, right=1in]{geometry}
\usepackage{times, amsthm, amssymb, amsmath, amsfonts, 
amsthm,mathrsfs,mathtools,bm}
\usepackage{microtype}
\usepackage{hyperref}
\usepackage[usenames,dvipsnames]{color}
\usepackage{enumerate}
\usepackage{enumitem}
\usepackage[capitalise]{cleveref}
\usepackage{xspace}

\newtheorem{theorem}{Theorem}

\newtheorem{proposition}[theorem]{Proposition}
\newtheorem{corollary}[theorem]{Corollary}

\newcommand{\N}{\mathbb{N}} 
\newcommand{\Z}{\mathbb{Z}} 
\newcommand{\R}{\mathbb{R}} 
\newcommand{\field}{\mathbb{K}} 
\newcommand{\genset}{A} 
\newcommand{\matA}{\mathbf{A}}
\newcommand{\matF}{\mathbf{F}} 
\newcommand{\matG}{\mathbf{G}}
\newcommand{\matH}{\mathbf{H}} 
\newcommand{\matM}{\mathbf{M}}
\newcommand{\matI}{\mathbf{I}}
\newcommand{\veca}{\mathbf{a}} 
\newcommand{\vecb}{\mathbf{b}}
\newcommand{\vecc}{\mathbf{c}} 
\newcommand{\vecd}{\mathbf{d}}
\newcommand{\vecg}{\mathbf{g}}
\newcommand{\vecu}{\mathbf{u}}
\newcommand{\vecv}{\mathbf{v}}
\newcommand{\vecw}{\mathbf{w}}
 
\newcommand{\sage}{\texttt{SageMath}\xspace}
\newcommand{\macaulay}{\texttt{Macaulay2}\xspace}
\newcommand{\libraryname}{\texttt{stdPairs.spyx}\xspace}
\newcommand{\normaliz}{\texttt{Normaliz}\xspace}

\DeclareMathOperator{\std}{std} 
\DeclareMathOperator{\stdp}{Std}

\makeatletter
\@namedef{subjclassname@2020}{
  \textup{2020} Mathematics Subject Classification}
\makeatother

\begin{document}
\title{Standard pairs of monomial ideals over non-normal affine semigroups in \sage}
\author{Byeongsu Yu}
\address{Mathematics Department\\Texas A\&M University\\College Station, TX 77843}
\email{byeongsu.yu@math.tamu.edu}

\begin{abstract}
We present \libraryname, a \sage library to compute standard pairs of
a monomial ideal over a pointed (non-normal) affine semigroup
ring. Moreover, \libraryname provides the associated prime ideals,
the corresponding multiplicities, and an irredundant irreducible
primary decomposition of a 
monomial ideal. The library expands on the \texttt{standardPairs}
function on \macaulay over polynomial rings, and is based on
algorithms from \cite{STDPAIR}. We also provide methods that allow the
outputs from this library to be compatible with the \normaliz package of \macaulay and \sage. 
\end{abstract}


\maketitle 

\section{Introduction}
\label{section:intro}

Affine semigroup rings are objects of many studies in combinatorial
commutative algebra. The goal of this article is to present the
\sage library \libraryname, which systematizes computations for
monomial ideals in affine semigroup 
rings. The algorithms implemented here are based on the notion of
\emph{standard pairs}, introduced for monomial ideals in polynomial
rings by \cite{MR1339920}, and generalized to the semigroup ring case in
\cite{STDPAIR}. Standard pairs are a combinatorial structure that contains
information on primary and irreducible decompositions of monomial
ideals, as well as multiplicities. One of the main contributions of
\cite{STDPAIR} is that standard pairs and the associated algebraic concepts
can be effectively computed over affine semigroup rings.

The \sage library \libraryname implements the algorithms of
\cite{STDPAIR} to calculate standard pairs for monomial ideals in any
pointed (non-normal) affine semigroup ring. This library can be
regarded as a generalization of \texttt{standardPairs} function in
\macaulay implemented by \cite{MR1949549}. This library can be obtained via \url{https://github.com/byeongsuyu/StdPairs}.

\subsection*{Outline}
\cref{section:classes} provides background on affine semigroup rings,
their monomial ideals, and related combinatorial notions. It also explains their implementation as \sage classes. \cref{section:algorithms} presents the implementation of an algorithm finding standard pairs, proposed in~\cite[Section 4]{STDPAIR}. \cref{section:compatibility} shows compatibility with the \normaliz package by introducing methods to translate objects in \sage into objects in \macaulay using \normaliz.

\subsection*{Acknowledgements}
We would like to express our deepest appreciation to Laura Matusevich for conversations and helpful comments on draft versions of this paper. Also, we are grateful to Matthias K{\"o}ppe for advice on using \texttt{zsolve} in \texttt{4ti2}. 

\subsection{Notation}
In this paper, a semigroup of nonnegative integers including 0 is denoted by $\N=\{0,1,2, \cdots \}$. A ring of integers is $\Z$.  All nonnegative real numbers are represented by $\R_{\geq 0}$. Boldface upper case letters and boldface lower case letters denote matrices and vectors, respectively. This will be used when we call a cone $\R_{\geq 0}\matA$ for some $d \times n$ matrix $\matA$ over $\Z$. An arbitrary field is $\field$.

\section{Affine semigroup, ideal, and proper pair as classes of \sage}
\label{section:classes}

\subsection{Mathematical Background}
An \emph{affine semigroup} is a semigroup of $\Z^{d}$ generated by
finitely many vectors $\veca_{1},\cdots, \veca_{n}$ for some $n \in
\N.$ We let $\matA$ be a $d \times n$ matrix whose column vectors are
$\veca_{1},\cdots, \veca_{n} \in \Z^{d}$. The set of all nonnegative
integer linear combinations of $\veca_{1},\cdots, \veca_{n}$, denoted
by $\N\matA$, is an affine semigroup. These columns $\veca_{1},\cdots,
\veca_{n}$ are the \emph{generators} of the affine semigroup $\N
\matA$; the matrix $\matA$ is called the \emph{generating
  matrix}. Since $\N \matA$ contains 0, an affine semigroup is a
commutative monoid. Given a field $\field$, we are concerned with the
\emph{affine semigroup ring} $\field[\N \matA]$. A natural first
example is the polynomial ring in $d$-variables; in this case, $\matA$
is the $d\times d$ identity matrix. We refer to \cite[Section~7]{CCA}
for more background on this topic. Throughout this article, we assume
that the affine semigroup $\N\genset$ under consideration is
\emph{pointed}, which means that the cone $\R_{\geq 0} \matA$ does not contain lines. 

An \emph{ideal} of an affine semigroup is a set $I \subset \N \matA$
such that $I + \N \matA \subseteq I$. There is a one-to-one
correspondence between monomial ideals of $\field[\N \matA]$ and
ideals of $\N \matA$. Therefore, the definition of prime, irreducible,
and primary ideals of $\field[\N \matA]$ can be naturally extended to
the ideals of an affine semigroup. The \emph{standard monomials} of an
ideal $I \subset \N \matA$ are all elements of $\N \matA \smallsetminus I$. Let $\std(I)$ be a set of all standard monomials with respect to $I$.

A \emph{face} of an affine semigroup $\N\matA$ is a subsemigroup
$\N\matF \subseteq \N\matA$ such that the complement $\N\matA \smallsetminus
\N\matF$ is an ideal of $\N\matA$ \cite[Definition 7.8]{CCA}. Equivalently,
it is a subsemigroup $\N\matF$ with the property
that $\veca +\vecb \in \N \matF$ if and only if $\veca,\vecb \in \N
\matF$. The faces of an affine semigroup form a lattice which is isomorphic to the face
lattice of a (real) cone over the affine
semigroup~\cite{MR1251956,CCA}. Thus, we may represent a face
$\N\matF$ as a submatrix $\matF$ of $\matA$.

A \emph{pair} is a tuple $(\veca, \matF)$ of an element $\veca$ in $\N
\matA$ and a face $\matF$ of $\N \matA$~\cite{STDPAIR}. A \emph{proper
  pair} of an ideal $I$ is a pair $(\veca, \matF)$ such that $\veca +
\N\matF \subseteq \std(I)$. A pair $(\veca, \matF)$ \emph{divides}
$(\vecb,\matG)$ if there exists $\vecc \in \N \matA$ such that $\veca
+ \vecc + \N \matF \subseteq \vecb + \N \matG$~\cite{STDPAIR}. The set
of all proper pairs of an ideal $I$ is partially ordered $\prec$ by
inclusion. In other words, $(\veca, \matF) \prec (\vecb, \matG)$ if
$\veca+ \N \matF \subset \vecb+ \N \matG$. The \emph{standard pairs}
of an ideal $I$ are the maximal elements of the set of all proper
pairs of $I$ in this partial order. We denote by $\stdp(I)$ the set of
all standard pairs of an ideal $I$. 

We remark that our notation here differs from existing notation for standard pairs over polynomial rings. Over the polynomial ring $\field[x_{1},x_{2},\cdots, x_{n}]$, a pair is a tuple $(x^{\veca},V)$ where $x^{\veca}$ is a monomial $x_{1}^{a_{1}}x_{2}^{a_{2}}\cdots x_{n}^{a_{n}}$ for some integer vector $\veca=(a_{1},a_{2},\cdots, a_{n})$ and $V$ is a set of variables \cite{MR1949549, MR1339920}. From the viewpoint of affine semigroup rings, the polynomial ring is a special case when the underlying affine semigroup is generated by an $n\times n$ identity matrix $\matI$. Since the cone $\R_{\geq0}\matI$ is a simplicial cone, i.e., every subset of rays form a face, we may interpret $V$ as a face. The following example shows the different notations for the standard pairs of a monomial ideal $I = \langle x^{\left[\begin{smallmatrix} 1\\ 3 \\ 1 \end{smallmatrix}\right]}, x^{\left[\begin{smallmatrix} 1\\ 2 \\ 2 \end{smallmatrix}\right]},x^{\left[\begin{smallmatrix} 0\\ 3 \\ 2 \end{smallmatrix}\right]}, x^{\left[\begin{smallmatrix} 0\\ 2 \\ 3 \end{smallmatrix}\right]} \rangle$ in the polynomial ring $\field[x_{1},x_{2},x_{3}]$. 

In \macaulay,
\begin{verbatim}
i1 : R = QQ[x,y,z];
i2 : I = monomialIdeal(x*y^3*z, x*y^2*z^2, y^3*z^2, y^2*z^3)
                       3      2 2   3 2   2 3
o2 = monomialIdeal (x*y z, x*y z , y z , y z )

o2 : MonomialIdeal of R

i3 : standardPairs I           
o3 = {{1, {x, z}}, {y, {x, z}}, {1, {x, y}}, {z, {y}}, 
  2           2 2
{y z, {x}}, {y z , {}}}

o3 : List
\end{verbatim}
whereas in the given library \libraryname in \sage,
\begin{verbatim}
sage: load("~/stdPairs.spyx")
Compiling /Users/byeongsuyu/stdPairs.spyx...
sage: A = matrix(ZZ,[[1,0,0],[0,1,0],[0,0,1]])
sage: Q = affineMonoid(A)
sage: M = matrix(ZZ,[[1,1,0,0],[3,2,3,2],[1,2,2,3]])
sage: I = monomialIdeal(Q,M)
sage: I.standardCover()
{(): [([[0], [2], [2]]^T,[[], [], []])],
 (1,): [([[0], [0], [1]]^T,[[0], [1], [0]])],
 (0,): [([[0], [2], [1]]^T,[[1], [0], [0]])],
 (0, 2): [([[0], [1], [0]]^T,[[1, 0], [0, 0], [0, 1]]),
  ([[0], [0], [0]]^T,[[1, 0], [0, 0], [0, 1]])],
 (0, 1): [([[0], [0], [0]]^T,[[1, 0], [0, 1], [0, 0]])]}
\end{verbatim}

\subsection{Classes in \libraryname}
\label{subsec:class_aff_monoid}

\sloppy We implement three classes related to affine semigroup, semigroup ideal, and proper pair respectively. This implementation is based on \sage 9.1 with \texttt{Python} 3.7.3. and \texttt{4ti2} package.

\subsubsection{Class \texttt{affineMonoid}}
This class is constructed by using an integer matrix $\matA$. The name follows the convention of \sage which distinguishes monoid from semigroup. In \sage, $\matA$ can be expressed as a 2-dimensional \texttt{Numpy} array or an integer matrix of \sage. For example,

\begin{verbatim}
sage: A = matrix(ZZ,[[1,2],[0,2]])
sage: load("stdPairs.spyx")
Compiling ./stdPairs.spyx...
sage: Q = affineMonoid(A)
\end{verbatim}

generates $Q$ as a type of \texttt{affineMonoid}. This class has several member variables and member functions as explained below.

\begin{itemize}[leftmargin=*]
\item \texttt{Q.gens} stores a matrix generating an affine monoid $Q$ as \texttt{Numpy.ndarray} type. This may not be a minimal generating set of $Q$.
\item \texttt{Q.mingens} stores a minimal generating matrix of an affine monoid of $Q$.
\item \texttt{Q.poly} is a real cone $\R_{\geq 0}Q$ represented as a type of \texttt{Polyhedron} in \sage. If one generates $Q$ with \texttt{True} parameter, i.e.,
\begin{verbatim}
sage: A = matrix(ZZ,[[1,2],[0,2]])
sage: load("stdPairs.spyx")
Compiling ./stdPairs.spyx...
sage: Q = affineMonoid(A,True)
\end{verbatim}
then \texttt{Q.poly} is of a class of \normaliz integral polyhedron. This requires \texttt{PyNormaliz} package. See~\cite{NormalizSage} for more details.
\item \texttt{Q.faceLattice} is a finite lattice containing all faces of the affine semigroup. A face in the lattice is saved as a tuple storing column numbers of generators $\matA$. \texttt{Q.faceLattice} is of type of \texttt{Finite Lattice Poset} in \sage. For example, 
\begin{verbatim}
sage: Q.faceLattice
Finite lattice containing 5 elements with distinguished linear 
extension
sage: Q.faceLattice.list()
[(-1,), (), (0,), (1,), (0, 1)]
\end{verbatim}
\item \texttt{Q.indexToFace} is a \texttt{dictionary} type whose key is a face as a tuple, and whose item is a face as an element of \texttt{Q.poly}. For example,
\begin{verbatim}
sage: Q.indexToFace
{(-1,): A -1-dimensional face of a Polyhedron in ZZ^2,
 (): A 0-dimensional face of a Polyhedron in ZZ^2 
 defined as the convex hull of 1 vertex,
 (0,): A 1-dimensional face of a Polyhedron in ZZ^2 
 defined as the convex hull of 1 vertex and 1 ray,
 (1,): A 1-dimensional face of a Polyhedron in ZZ^2 
 defined as the convex hull of 1 vertex and 1 ray,
 (0,1): A 2-dimensional face of a Polyhedron in ZZ^2 
  defined as the convex hull of 1 vertex and 2 rays}
\end{verbatim}
\item \texttt{Q.integralSupportVectors} is a \texttt{dictionary} type whose key is a face $\matF$ as a tuple and whose item is a set of the integral support functions of facets containing $\matF$ as a vector form. An \emph{integral support function} $\phi_{\matH}$ of a facet $\matH$ is a linear function $\phi_{\matH}:\R^{d} \to \R$ such that $\phi_{\matH}(\Z^{d}) =\Z$, $\phi_{\matH}(\veca) \geq 0$ for all column vectors $\veca$ of generators $A$, and $\phi_{\matH}(\veca) =0$ if and only if $\veca \in \matH$. By linearity, $\phi_{\matH}(\veca)=\vecb \cdot \veca$ for some rational vector $\vecb$. We call $\vecb$ as an \emph{integral support vector}. Each item of \texttt{Q.integralSupportVectors} returns a matrix whose rows are integral support vectors of facets containing the key.
For example, 
\begin{verbatim}
sage: Q.integralSupportVectors
{(): array([[ 0,  1],
        [ 1, -1]]),
 (0,): array([[0, 1]]),
 (1,): array([[ 1, -1]]),
 (0, 1): array([], dtype=float64)}
\end{verbatim}
See~\cite[Definition 2.1]{STDPAIR} for the precise definition of a (primitive) integral support function.
\item \texttt{Q.isEmpty()} returns a boolean value indicating whether $Q$ is a trivial affine semigroup or not. A \emph{trivial affine semigroup} is an empty set as an affine semigroup.
\item \texttt{Q.isPointed()} returns a boolean value indicating whether $Q$ is a pointed affine semigroup or not.
\item \texttt{Q.isElement(vector $\vecb$)} returns nonnegative integral inhomogeneous solutions (minimal integer solutions) of $\matA  \mathbf{x} = \vecb$ using \texttt{zsolve} in~\cite{4ti2}. If $\vecb$ is not an element of an affine semigroup $Q$, then it returns an empty matrix.
\item \texttt{Q.IntersectionOfPairs(vector  $\veca$, tuple $F$, vector $\vecb$, tuple $G$)} gives nonnegative integral inhomogeneous solutions (minimal integer solutions) of $$\left[\begin{smallmatrix} \matF & -\matG \end{smallmatrix}\right] \left[\begin{smallmatrix} \mathbf{u} \\ \mathbf{v} \end{smallmatrix}\right] = \vecb-\veca$$ using \texttt{zsolve} in~\cite{4ti2}, where $\matF$ and $\matG$ are submatrices of $A$ corresponding to its tuple forms respectively. This equation has solutions if and only if two pairs $(\veca, \matF)$ and $(\vecb, \matG)$ have a common element. This returns an empty matrix if it has no solution.
\item \texttt{Q.Face(tuple index)} returns a face as a submatrix of a generator $\matA$ corresponding to a given tuple \texttt{index}. For example,
\begin{verbatim}
sage: Q.Face((1,))
array([[2],
       [2]])
\end{verbatim}
\item \texttt{Q.IndFace(matrix face)} returns a face as a tuple of indices of column vectors of a generator $\matA$ corresponding to a given submatrix \texttt{face} of $\matA$. For example,
\begin{verbatim}
sage: M = matrix(ZZ,[[2],[2]])
sage: Q.IndFace(M)
(1,)
\end{verbatim}
\item \texttt{Q.primeIdeal(tuple face)} returns a prime ideal corresponding to the face represented by \texttt{face} as an object of \texttt{monomialIdeal}.
\begin{verbatim}
sage: Q.primeIdeal((1,))
An ideal whose generating set is 
[[1]
 [0]]
\end{verbatim}
\item \texttt{Q.save(string path)} saves the given \texttt{affineMonoid} object as a text file. This can be loaded again using \texttt{load\_stdPairs(string path)}, which will be explained in \cref{subsec:global_methods}. If the save is successful, then it returns 1.
\item \texttt{Q.hashstring} is a string unique for the (mathematically) same affine monoid.
\end{itemize}

Moreover, one can directly compare affine semigroups using the equality operator \texttt{==} in \sage.

\subsubsection{Class \texttt{monomialIdeal}}
\label{subsec:class_ideal}

This class is constructed by an affine semigroup $Q$ and generators of an ideal as a matrix form, say $\matM$, which is a 2-dimensional \texttt{Numpy} array or an integer matrix of \sage. For example,

\begin{verbatim}
sage: M = matrix(ZZ,[[4,6],[4,6]])
sage: I = monomialIdeal(Q,M)
sage: I
An ideal whose generating set is 
[[4]
 [4]]
\end{verbatim}

As shown in the example above, this class stores only minimal generators of the ideal. The member variables and functions are explained below.
\begin{itemize}[leftmargin=*]
\item \texttt{I.gens} shows a (minimal) generators of $I$ as a \texttt{Numpy} array form. 
\item \texttt{I.ambientMonoid} stores the ambient affine semigroup of $I$.
\item \texttt{I.isPrincipal()} returns a boolean value indicating whether $I$ is principal or not. Likewise, \texttt{I.isEmpty()}, \texttt{I.isIrreducible()}, \texttt{I.isPrimary()}, \texttt{I.isPrime()}, and \texttt{I.isRadical()} return a boolean value indicating whether $I$ has the properties implied by their name or not.
\item \texttt{I.isElement(vector $\vecb$)} returns nonnegative integral inhomogeneous solutions (minimal integer solutions) of $\matA  \mathbf{x} = \vecb-\veca$ for each generator $\veca$ of $I$ using \texttt{zsolve} in~\cite{4ti2}. If $\vecb$ is an element of ideal, then it returns a list $[\mathbf{x}, \veca]$ for some generator $\veca$ such that $\veca + \matA \mathbf{x}^{T}=\vecb$. Otherwise, it returns an empty matrix.
\item \texttt{I.isStdMonomial(vector $\vecb$)} returns a boolean value indicating whether the given vector $\vecb$ is a standard monomial or not.
\item \texttt{radicalIdeal(I)} returns the radical of $I$ as an \texttt{monomialIdeal} object.
\item \texttt{I.standardCover()} returns the \emph{standard cover} of $I$. The definition of standard cover and algorithms will be given in \cref{subsec:principal_ideal}.
\item \texttt{I.overlapClasses()} returns the overlap classes of $I$. An \emph{overlap class} of an ideal $I$ is a set of standard pairs such that their representing submonoids intersect nontrivially. Moreover, \texttt{I.maximalOverlapClasses()} returns all maximal overlap classes of $I$. See~\cite[Section 3]{STDPAIR} for the detail.
\item \texttt{I.associatedPrimes()} returns all associated prime ideals of $I$ as a \texttt{dictionary} type in Python. In other words, the function returns a dictionary whose keys are faces of the affine semigroup as \text{tuple} form and whose value is a list containing associated prime ideals corresponding to the face in its key.
\item \texttt{I.multiplicity(ideal P or face $\matF$)} returns a multiplicity of an associated prime $P$ over the given ideal $I$. Also, the method takes the face $\matF$ (as a tuple) corresponding to a prime ideal $P$ as input instead.
\item \texttt{I.irreducibleDecomposition()} returns the irredundant irreducible primary decomposition of $I$ as a list. Since it takes a lot of time, this library also provides a method which calculates only one irreducible primary component corresponding to a maximal overlap class. This class is \texttt{I.irreducibleComponent(tuple face, list ov\_class)}, where \texttt{ov\_class} is a maximal overlap class of $I$ and \texttt{face} is a face corresponding to the overlap class \texttt{ov\_class}.
\item \texttt{I.intersect(J)} returns an intersection of two ideals $I$ and $J$ as a \texttt{monomialIdeal} object. Likewise, addition \texttt{+}, multiplication \texttt{$\ast$}, and comparison \texttt{==} are defined between two objects. The following example shows an addition of two monomial ideals in \sage.
\begin{verbatim}
sage: I = monomialIdeal(Q,matrix(ZZ,[[4,6],[4,6]]))
sage: J = monomialIdeal(Q,matrix(ZZ,[[5],[0]]))
sage: I.intersect(J)
An ideal whose generating set is 
[[9]
 [4]]
sage: I+J
An ideal whose generating set is 
[[5 4]
 [0 4]]
\end{verbatim}
\item \texttt{I.save(string path)} saves the given \texttt{monomialIdeal} object $I$ as a text file. Especially, it saves not only the generators of $I$, but also its standard cover, overlap classes, associated primes, and irreducible primary decompositions if they were calculated. This can be loaded again using \texttt{load\_stdPairs(string path)}, which will be explained in \cref{subsec:global_methods}. If the save is successful, then it returns 1.
\item \texttt{I.hashstring} is a string unique for the (mathematically) same ideal.
\end{itemize}

\subsubsection{Class \texttt{properPair}}
\label{subsec:class_pair}

A proper pair $(\veca, \matF)$ of an ideal $I$ can be declared in \sage by specifying an ideal $I$, a standard monomial $\veca$ as a matrix form (or \texttt{Numpy} 2D array), and a face $\matF$ as a tuple. If $(\veca, \matF)$ is not proper, then \sage calls a \texttt{ValueError}. The following example shows two ways of defining a proper pair. 

\begin{verbatim}
sage: I = monomialIdeal(Q,matrix(ZZ,[[4,6],[4,6]]))
sage: PP = properPair(np.array([2,0])[np.newaxis].T,(0,),I)
sage: PP
([[2], [0]]^T,[[1], [0]])
sage: QQ = properPair(np.array([2,0])[np.newaxis].T,(0,),I, True)
sage: QQ
([[2], [0]]^T,[[1], [0]])
\end{verbatim}

The second line tests whether the pair is a proper pair of the given ideal $I$ or not before generating \texttt{PP}. However, the fourth line generates \texttt{QQ} without such a test. Use the third parameter with \texttt{True} only if the generating pair is proper a priori. In any case, each \texttt{PP} and \texttt{QQ} denotes proper pair whose initial monomial is $\left[\begin{smallmatrix}2 \\ 0\end{smallmatrix}\right]$ and whose face is $\left[\begin{smallmatrix}1 \\ 0\end{smallmatrix}\right]$. 

The member variables and functions are explained below. We assume that \texttt{PP} denotes a proper pair $(\veca,\matF)$.

\begin{itemize}[leftmargin=*]
\item \texttt{PP.monomial}, \texttt{PP.face}, and \texttt{PP.ambientIdeal} return the initial monomial $\veca$ (as \texttt{Numpy} 2D array), the face $\matF$ (as a \texttt{tuple}), and its ambient ideal (as an object of \texttt{affineMonoid}) respectively.
\item \texttt{PP.isMaximal()} returns a boolean value indicating whether the given pair is maximal with respect to the divisibility of proper pairs of the ambient ideal.
\item \texttt{PP.isElement(vector $\vecb$)} returns nonnegative integral inhomogeneous solutions (minimal integer solutions) of $\veca +\matF  \mathbf{x} = \vecb$ using \texttt{zsolve} in~\cite{4ti2}. If $\vecb$ is not an element of the submonoid $\veca + \N\matF$, then it returns an empty matrix.
\item \texttt{divides(pair PP, pair QQ)} returns a matrix whose row $\mathbf{u}$ is a minimal solution of $\veca+\matA\mathbf{u}+\N\matF = \vecb+\N\matG$ if $PP = (\veca, \matF)$ and $QQ = (\vecb, \matG)$. The returned value is a nonempty matrix if and only if a pair $PP$ divides a pair $QQ$. For example,
\begin{verbatim}
sage: I = monomialIdeal(Q,matrix(ZZ,[[4,6],[4,6]]))
sage: PP = properPair(matrix(ZZ,[[2],[0]]),(0,),I)
sage: PP
([[2], [0]]^T,[[1], [0]])
sage: QQ = properPair(matrix(ZZ,[[2],[0]]),(0,),I,True)
sage: QQ
([[2], [0]]^T,[[1], [0]])
sage: divides(PP,QQ)
[0 0 0]
\end{verbatim}
since $PP=QQ$. 
\item Like \texttt{affineMonoid} or \texttt{monomialIdeal}, one can directly compare proper pairs using the equality operator \texttt{==} in \sage.
\item \texttt{PP.hashstring} is a string unique for the (mathematically) same proper pair.
\end{itemize}

\subsection{Global methods}
\label{subsec:global_methods}

Global methods are introduced below.
\begin{itemize}[leftmargin=*]
\item \texttt{unique\_monoids(list L)}, \texttt{unique\_ideals(list L)}, \texttt{unique\_pairs(list L)} and \texttt{unique\_np\_arrays(list L)} receive a list of variables whose type is \texttt{affineMonoid},  \texttt{monomialIdeal}, \texttt{properPair} and \texttt{numpy} 2D array respectively, and return a list of (mathematically) distinct objects.
\item \texttt{save\_cover(dict cover, affineMonoid Q, string path)} saves a \emph{cover}, a \texttt{dict} type object whose keys are faces (as \texttt{tuple} objects) of $Q$ and whose values are a list of pairs (as \texttt{properPair} objects) into a file located in \texttt{path}. Note that all pairs are saved as proper pairs under the empty ideal. If the save is successful, then it returns 1.
\item \texttt{load\_stdPairs(string path)} loads an \texttt{affineMonoid} object, an \texttt{monomialIdeal} object, or a cover stored in a file located in \texttt{path}. It is useful for users who want to avoid repeating calculation which was previously done. For example, 
\begin{verbatim}
sage: I = load_stdPairs("/Users/byeongsuyu/ex_ideal.txt")
sage: I.ambientMonoid
An affine semigroup whose generating set is 
[[1 1 2 3]
 [1 2 0 0]]
sage: I
An ideal whose generating set is 
[[3 5 6]
 [2 1 1]]
 sage: I.irreducibleDecomposition()
[An ideal whose generating set is 
 [[3 4 2 3 5]
  [2 0 4 4 0]], An ideal whose generating set is 
 [[2 3]
  [0 0]], An ideal whose generating set is 
 [[1 1]
  [1 2]]]
sage:
\end{verbatim}
shows that the library loads pre-calculated irreducible decomposition of the given ideal. This save file can be obtained via \url{https://github.com/byeongsuyu/StdPairs}.
\end{itemize}

\section{Implementation of an algorithm finding standard pairs}
\label{section:algorithms}

\subsection{Case 1: Principal Ideal}
\label{subsec:principal_ideal}

A \emph{cover} of standard monomials of an ideal $I$ is a set of
proper pairs of $I$ such that the union of all subsemigroup
$\veca+\N\matF$ corresponding to an element $(\veca,\matF)$ of the
cover is equal to the set of all standard monomials. The
\emph{standard cover} of an ideal $I$ is a cover of $I$ whose elements
are standard pairs. The standard cover of a monomial ideal $I$ is
unique by the maximality of standard pairs among all proper pairs of
$I$. A key idea in~\cite[Section 4]{STDPAIR} is to construct covers
containing all standard pairs. Once  a cover is obtained, we can  then
produce the standard cover. 

The following result helps to compute the standard cover in the
special case of a principal ideal.
\begin{theorem}[{{\cite[Theorem 4.1]{STDPAIR}}}]

\label{thm:pair_difference}
  Let $\vecb, \vecb' \in \N \matA$ and let $\matG, \matG'$ be faces of $A$ such that
  $\matG\cap \matG' = \matG$. There exists an algorithm to compute a finite collection $C$ of pairs over faces of $\matG$ such that
  \[
    (\vecb+\matG) \smallsetminus (\vecb'+\matG') = \cup_{(\veca,\matF) \in C} (\veca+\matF).
  \]
\end{theorem}

The \emph{pair difference} of the pairs $(\vecb, \matG)$ and $(\vecb', \matG')$ is a finite collection of pairs over faces of $\matG$ given by \cref{thm:pair_difference}. 

\begin{corollary}
\label{cor:principal_ideal_case}
Given a principal ideal $I = \langle \vecb \rangle$, the pair difference of pairs $(0, \matA)$ and $(\vecb, \matA)$ is the standard cover of $I$.
\end{corollary}

\begin{proof}
\cref{thm:pair_difference} implies that the pair difference is a cover of $I$. To see it is the standard cover, suppose that the ambient affine semigroup is generated by  $\matA = \left[\begin{smallmatrix} \veca_{1}& \cdots &\veca_{n}\end{smallmatrix}\right].$ Let $(\vecc, F)$ be a proper pair in the pair difference. Without loss of generality, we assume that $F = \left[\begin{smallmatrix} \veca_{1}& \cdots &\veca_{m}\end{smallmatrix}\right]$  for some $m<n$ by renumbering indices. By the proof of \cref{thm:pair_difference} in~\cite{STDPAIR}, $\vecc = A \cdot \vecu$ where $x^{\vecu} \in \field[\N^{n}]$ is a standard monomial such that $(x^{\vecu}, \{ x_{1}, \cdots, x_{m}\})$ is a standard pair of some monomial ideal $J$ in $\field[\N^{n}]$. 

Suppose that there exists $(\vecd, \matG)$ such that $F \subseteq \matG$ and $\vecd + \vecg = \vecc$ for some $\vecg \in G$. Since $\vecd \in \N \matA $, $\vecd= A \vecw$ for some $\vecw \in \N^{n}$. Since $A$ is pointed, $\vecw$ is coordinatewise less than $\vecu$. Thus, $(x^{\vecw}, \{ x_{1}, \cdots, x_{m}\}  )$ contains $(x^{\vecu}, \{ x_{1}, \cdots, x_{m}\})$. Lastly, $(x^{\vecw}, \{ x_{1}, \cdots, x_{m}\})$ is a proper pair of $J$, otherwise, there exists $x^{\vecv} \in \field[x_{1}, \cdots, x_{m}] \subseteq \field[\N^{n}]$ such that $x^{\vecw+\vecv} \in J$. Then, $x^{\vecg}x^{\vecw+\vecv} \in J \implies x^{\vecu+ \vecv} \in J \cap (x^{\vecu}, \{ x_{1}, \cdots, x_{m}\}) = \emptyset$ leads to a contradiction.

Thus, by maximality of the standard pair, $\vecw = \vecu$. This implies $\vecd = \vecc$. Moreover, $G=F$, otherwise there exists $j \in  \{ 1,2,\cdots, n\} \smallsetminus \{ 1,\cdots, m\}$ such that $x^{\vecu}x_{j}^{l} \not\in J$ for any $l$, which implies that $(x^{\vecu}, \{ x_{1}, \cdots, x_{m},x_{j}\})$ is a proper pair of $J$ strictly containing a standard pair $(x^{\vecu}, \{ x_{1}, \cdots, x_{m}\})$ of $J$, a contradiction.
\end{proof}

\cref{thm:pair_difference} is implemented as a method \texttt{pair\_difference(($\vecb,\matF$), ($\vecb',\matF'$))} in the library \libraryname. Two input arguments should be the type of \texttt{properPair}. It returns the pair difference of pairs $(\vecb,\matF)$ and $(\vecb',\matF')$ with \texttt{dictionary} type, called \texttt{Cover}. \texttt{Cover} classifies pairs by its faces. For example, the code below shows the pair difference of pairs $(0, \matA)$ and $((0,2),\matA)$, which are 

$$ \left( 0, \begin{bmatrix} 2 \\ 0 \end{bmatrix}\right), \left( \begin{bmatrix} 0 \\ 1\end{bmatrix}, \begin{bmatrix} 2 \\ 0 \end{bmatrix}\right), \left( \begin{bmatrix} 1 \\ 2\end{bmatrix}, \begin{bmatrix} 2 \\ 0 \end{bmatrix}\right) \text{, and } \left( \begin{bmatrix} 1 \\ 1\end{bmatrix}, \begin{bmatrix} 2 \\ 0 \end{bmatrix}\right)$$

\begin{verbatim}
sage: B = affineMonoid(matrix(ZZ, [[2,0,1],[0,1,1]]))
sage: I = monomialIdeal(B, matrix(ZZ,0))
sage: C= properPair(np.array([[0,0]]).T, (0,1,2), I ) 
sage: D= properPair(np.array([[0,2]]).T, (0,1,2), I ) 
sage: print(pair_difference(C,D))
{(0,): [([[1], [1]]^T,[[2], [0]]), ([[1], [2]]^T,[[2], [0]]), 
([[0], [1]]^T,[[2], [0]]), ([[0], [0]]^T,[[2], [0]])]}
\end{verbatim}

By \cref{cor:principal_ideal_case}, it is a set of all standard pairs of an ideal $I = \langle (0,2)\rangle$ in an affine semigroup $\N \left[\begin{smallmatrix} 2& 0 & 1 \\ 0 & 1 & 1\end{smallmatrix}\right]$.

\texttt{pair\_difference(($\vecb,\matF$), ($\vecb',\matF'$))} uses \texttt{standardPairs} of \macaulay internally to find standard pairs of a polynomial ring, which is implemented by~\cite{MR1949549}. Briefly, this method \texttt{pair\_difference(($\vecb,\matF$), ($\vecb',\matF'$))} calculates minimal solution of the integer linear system 
$$ \begin{bmatrix} F & -F'\end{bmatrix} \begin{bmatrix} \vecu \\ \vecv\end{bmatrix} = \vecb'-\vecb$$
using \texttt{zsolve} in \texttt{4ti2}. The solutions form an ideal $J$ of a polynomial ring in the proof of \cref{thm:pair_difference} on \macaulay. \texttt{standardPairs} derives standard pairs of $J$. Lastly, the method \texttt{pair\_difference} constructs proper pairs based on the standard pairs of $J$, and classify the proper pairs based on their faces and return the pair difference. 

\subsection{Case 2: General ideal}
\label{subsec:general_ideal}

\cite[Proposition 4.4]{STDPAIR} gives an algorithm to find the standard cover of non-principal monomial ideals.
\begin{proposition}[{{\cite[Proposition 4.4]{STDPAIR}}}]
\label{prop:cover_to_std_pairs}
Let $I$ be a monomial ideal in $\field[\N \matA]$. There is an algorithm whose input is a cover of the standard monomials of $I$, and whose output is the standard cover of $I$.
\end{proposition}
According to the proof of \cite[Proposition 4.4]{STDPAIR}, this is achieved by repeating the procedures below.

\begin{enumerate}[leftmargin=*]
\label{enum:prop4.4}
\item Input: $C_{0}$, an initial cover of $I$.
\item \label{enum:zero_to_one} For each $(\veca, F) \in C_{0}$, find minimal solutions of $(\veca+\R F) \cap \N \matA$ using the primitive integral support functions. (See \cite[Lemma 4.2]{STDPAIR} for the detail.)
\begin{itemize}[leftmargin=*]
\item  If $\vecb_{1},\vecb_{2},\cdots, \vecb_{m}$ are minimal solutions of $(\veca+\R F) \cap \N \matA$, construct pairs such as $(\vecb_{1},F), (\vecb_{2},F), \cdots, (\vecb_{m},F)$ and store them in the variable $C_{1}$.
\end{itemize}
\item\label{enum:one_to_two} For each pair $(\vecb, F) \in C_{1}$, construct $(\vecb, G)$ for any face $G$ which is not strictly contained in $F$. If $(\vecb, G)$ is a proper pair of $I$, save $(\vecb, G)$ on the variable $C_{2}$.
\item If $C_{0}$ is equal to $C_{2}$, done. Otherwise, set $C_{0}:=C_{2}$ and repeat the above process.
\end{enumerate}

\texttt{czero\_to\_cone($C_0$, $I$)} method in \libraryname implements \cref{enum:zero_to_one} to return $C_{1}$. It calls \texttt{minimalHoles(vector $\veca$, face F, affine semigroup A)} internally, which is the implementation of Lemma 4.2. \texttt{cone\_to\_ctwo($C_1$, $I$)} method implements \cref{enum:one_to_two}. Since the constructor function of the class \texttt{properPair} checks whether the pair is proper or not, the method \texttt{cone\_to\_ctwo($C_1$, $I$)} tries to construct proper pairs as a variable in \sage and record it if it is successful. Lastly, \texttt{cover\_to\_stdPairs( $C,I,$ loopNum=1000)} implements the whole procedure \cref{enum:prop4.4}. Though \cref{prop:cover_to_std_pairs} shows that the given procedure terminates with finitely many steps, we set \texttt{loopNum} for safety. \texttt{loopNum} determines how many steps the loop can use for finding standard pairs. The default value is 1000.

Now we are ready to find the standard cover of a general ideal $I$ whose minimal generators are $\langle \vecb_{1},\cdots, \vecb_{n}\rangle$, one can find standard pairs as in \cite[Theorem 4.5]{STDPAIR} described below.
\begin{enumerate}[leftmargin=*]
\item Find the standard cover $C$ of $\langle \vecb_{1}\rangle$ using pair difference.
\item For $i=2$ to $n$:
\begin{enumerate}[leftmargin=*]
\item For each pair $(\vecb, F)$ in $C$, replace it with elements of the pair difference of pairs $(\vecb,F)$ and $(\vecb_{i}, A)$. After this process $C$ is a cover of an ideal $\langle \vecb_{1},\vecb_{2},\cdots, \vecb_{i}\rangle$.
\item Using an algorithm of \cref{prop:cover_to_std_pairs} find the standard cover $C'$ of $\langle \vecb_{1},\vecb_{2},\cdots, \vecb_{i}\rangle$.
\item Replace $C$ with $C'$.
\end{enumerate}
\item Return $C$.
\end{enumerate}
The returned value $C$ is now the standard cover of $I$. 

\libraryname implements \cite[Theorem 4.5]{STDPAIR} as a method \texttt{standardPairs($I$)}. This method has an input $I$ whose type is \texttt{monomialIdeal}. It returns a cover whose type is \texttt{dictionary}, classifying standard pairs by its face. For example, the code below shows that the standard cover of an ideal generated by $\left[\begin{smallmatrix}2 & 2 & 2\\ 0 & 1 & 2 \\ 2 & 2 &2 \end{smallmatrix}\right]$ in an affine semigroup $\N \matA =\N \left[\begin{smallmatrix} 0&1&1&0 \\ 0&0&1&1 \\ 1&1&1&1\end{smallmatrix}\right]$ is
$$\left\{ \left( 0, \left[\begin{smallmatrix} 0 &0 \\ 0 & 1 \\ 1 & 1 \end{smallmatrix}\right]\right), \left( \left[\begin{smallmatrix} 1 \\ 1 \\ 1\end{smallmatrix}\right], \left[\begin{smallmatrix} 0 &0 \\ 0 & 1 \\ 1 & 1 \end{smallmatrix}\right]\right),  \text{ and }\left( \left[\begin{smallmatrix} 1 \\ 0 \\ 1\end{smallmatrix}\right], \left[\begin{smallmatrix} 0 &0 \\ 0 & 1 \\ 1 & 1 \end{smallmatrix}\right]\right) \right\}.$$
\begin{verbatim}
sage: Q=affineMonoid(matrix(ZZ,[[0,1,1,0],[0,0,1,1],[1,1,1,1]]))
sage: I=monomialIdeal(Q, matrix(ZZ,[[2,2,2],[0,1,2],[2,2,2]]))
sage: I.standardCover()
It takes a few minutes, depending on the system.
Cover for 1 generator was calculated.  2  generators are left. 
Cover for 2 generators was calculated.  1  generators are left. 
Cover for 3 generators was calculated.  0  generators are left. 
{(0, 3): [([[0], [0], [0]]^T,[[0, 0], [0, 1], [1, 1]]),
  ([[1], [1], [1]]^T,[[0, 0], [0, 1], [1, 1]]),
  ([[1], [0], [1]]^T,[[0, 0], [0, 1], [1, 1]])]}
\end{verbatim}

\section{Compatibility with \normaliz package in \sage and \macaulay}
\label{section:compatibility}

\normaliz is a package in \sage and \macaulay for finding Hilbert bases of rational cones and its normal affine monoid \cite{MR2659215}. \libraryname has methods translating classes in \cref{section:classes} into objects in the \normaliz package. If an affine semigroup $\N \matA$ is \emph{normal}, i.e., $\N \matA = \Z^{d} \cap \R_{\geq0}A$, then this translation works well. However, if it is not normal, then this translates into the saturation of the given affine semigroup described in \cref{section:classes}.

For \sage, one can have a polyhedron over $\Z$ with the \normaliz package in \sage by adding an argument \texttt{True} on the constructor of \texttt{affineMonoid}. For example, the code below gives an \texttt{affineMonoid} class variable $Q$ whose member variable $Q.poly$ is a polyhedron over $\Z$ with \normaliz.

\begin{verbatim}
sage: load("stdPairs.spyx")
Compiling ./stdPairs.spyx...
sage: Q=affineMonoid(matrix(ZZ, [[0,1,1,0],[0,0,1,1],[1,1,1,1]]),
 True)
\end{verbatim}

For \macaulay, \texttt{ToMacaulay2( monomialIdeal I)} returns a dictionary storing variable of \macaulay computations. This dictionary contains affine semigroup ring, a list of generators of an ideal, and a list of standard pairs in \macaulay. For example,
\begin{verbatim}
sage: load("stdPairs.spyx")
Compiling ./stdPairs.spyx...
sage: Q=affineMonoid(matrix(ZZ,[[0,1,1,0],[0,0,1,1],[1,1,1,1]]))
sage: I=monomialIdeal(Q, matrix(ZZ,[[2,2,2],[0,1,2],[2,2,2]]))
sage: S = ToMacaulay2(I,I.standardCover())
sage: print(S['AffineSemigroupRing'])
MonomialSubalgebra{cache => CacheTable{}             }
                   generators => {c, a*c, a*b*c, b*c}
                   ring => R
sage: print(S['MonomialIdeal'])
  2 2   2   2   2 2 2
{a c , a b*c , a b c }
sage: print(S['StandardCover']) 
{{a*c, {c, b*c}}, {a*b*c, {c, b*c}}, {1, {c, b*c}}}
\end{verbatim}
In \macaulay, a type \texttt{MonomialSubalgebra} in the \normaliz package may correspond to an affine semigroup ring. Since \normaliz has no attributes for a monomial ideal of the type \texttt{MonomialSubalgebra}, the ideal is stored as a list of its generators. The standard cover of $I$ is also sent to \macaulay as a nested list, similar to the output of the method \texttt{standardPairs} in \macaulay.

Conversely, \texttt{FromMacaulay( Macaulay2 $S$)} translates \texttt{monomialSubalgebra} object $S$ of \macaulay into an \texttt{affineMonoid} object in \libraryname. For example, 
\begin{verbatim}
sage: R = macaulay2('ZZ[x,y,z]')
sage: macaulay2("loadPackage Normaliz")
Normaliz
sage: S=macaulay2('createMonomialSubalgebra {x^2*y, x*z, z^3}')
sage: A=FromMacaulay(S)
sage: A
An affine semigroup whose generating set is 
[[2 1 0]
 [1 0 0]
 [0 1 3]]
\end{verbatim}

\bibliographystyle{plain}
\bibliography{std_pair_calculation}
\end{document}